\documentclass[12pt,reqno]{amsart}
\usepackage{fullpage,amssymb,amsmath,mathrsfs,enumerate}

\newtheorem{theorem}{Theorem}
\newtheorem{proposition}[theorem]{Proposition}
\newtheorem{corollary}[theorem]{Corollary}

\theoremstyle{definition}
\newtheorem*{definition}{Definition}

\theoremstyle{remark}
\newtheorem{claim}{Claim}

\newcommand\setsep{;\ }
\def\supp{\operatorname{supp}}
\def\restriction{\!\upharpoonright}
\def\er{\mathbb R}

\def\O{\mathcal{O}}
\def\N{\mathcal{N}}
\def\F{\mathcal F}
\def\B{\mathcal B}
\def\en{\mathbb N}
\def\C{\mathcal{C}}
\def \exp {\operatorname{exp}}

\begin{document}
\title{Characterization of compact monotonically ($\omega$)-monolithic spaces using system of retractions}
\author{Marek C\'uth}
\address{Department of Mathematical Analysis \\
Faculty of Mathematics and Physic\\ Charles University\\
Sokolovsk\'{a} 83, 186 \ 75\\Praha 8, Czech Republic}
\email{marek.cuth@gmail.com}
\thanks{The author was supported by the Research grant GA \v{C}R P201/12/0290.} 
\begin{abstract}We prove that a compact space is monotonically Sokolov if and only if it is monotonically $\omega$-monolithic. This gives answers to several questions of R. Rojas-Hern{\'a}ndez and V. V. Tkachuk.\end{abstract}
\maketitle
\section{Introduction}
Spaces with a rich family of retractions often occur both in topology and functional analysis. In Banach space theory, using systems of retractions we can obtain a system of projections and consequently find a Markushevich basis; see e.g. \cite{cuth}. In topology, Gul$'$ko used families of retractions in \cite{gul} to prove that a compact space $K$ is Corson whenever $\C_p(K)$ has the Lindel\"of $\Sigma$-property. The method of Gul$'$ko's proof was further studied and precised in \cite{tka07}.

One of the possible concepts of a family of retractions was recently introduced in \cite{tkaRoj}. Spaces having such a system were called monotonically Sokolov and using them, an answer to Problem 3.8 from \cite{tka05} was given.

In this note we give a positive answer to Question 6.3 from \cite{tkaRoj}, i.e. we prove that a compact space is monotonically Sokolov if and only if it is monotonically $\omega$-monolithic.

\begin{theorem}\label{tMain}Let $K$ be a compact space. Then the following conditions are equivalent:
	\begin{enumerate}[\upshape (i)]
		\item $K$ is monotonically monolithic
		\item $K$ is monotonically $\omega$-monolithic
		\item $K$ is monotonically Sokolov
	\end{enumerate}
\end{theorem}

As a consequence answers to Questions 6.4 and 6.5 from \cite{tkaRoj} easily follow.

\begin{corollary}\label{cMain1}If $K$ is a compact Collins-Roscoe space, then it is monotonically Sokolov.
\end{corollary}

\begin{corollary}\label{cMain2}There exists a compact space $K$ such that it is monotonically Sokolov but not Gul$'$ko.
\end{corollary}

\section{Preliminaries}
We denote by $\omega$ the set of all natural numbers (including $0$). If $X$ is a set then $\exp(X) = \{Y\setsep Y\subset X\}$.

All topological spaces are assumed to be Hausdorff. Let $T$ be a topological space. The closure of a set $A$ we denote by
$\overline{A}$. We denote the topology of $T$ by $\tau(T)$ and $\tau(x,T) = \{U\in\tau(T)\setsep x\in U\}$ for any $x\in T$. A family $\N$ of subsets of $T$ is an \emph{external network} of $A$ in $T$ if for any $a\in A$ and $U\in\tau(a,T)$ there exists $N\in\N$ such that $a\in N\subset U$.

Given an infinite cardinal $\kappa$ say that a space $T$ is \emph{monotonically $\kappa$-monolithic} if, to any set $A\subset T$ with $|A|\leq\kappa$, we can assign an external network $\O(A)$ to the set $\overline{A}$ in such a way that the following conditions are satisfied:
\begin{enumerate}[\upshape (i)]
	\item $|\O(A)|\leq |A| + \omega$;
	\item if $A\subset B\subset T$ and $|B|\leq\kappa$ then $\O(A)\subset \O(B)$;
	\item if $\lambda\leq\kappa$ is a cardinal and we have a family $\{A_\alpha\setsep \alpha < \lambda\}\subset [X]^{\leq\kappa}$ such that $\alpha < \beta < \lambda$ implies $A_\alpha\subset A_\beta$ then $\O(\bigcup_{\alpha < \lambda}A_\alpha) = \bigcup_{\alpha < \lambda}\O(A_\alpha)$.
\end{enumerate}
The space $T$ is \emph{monotonically monolithic} if it is monotonically $\kappa$-monolithic for any infinite cardinal $\kappa$.

Topological space $T$ is a \emph{Collins-Roscoe space} if for each $x\in T$, one can assign a countable family $\O(x)$ of subsets of $T$ such that, for any $A\subset T$, $\bigcup\{\O(x)\setsep x\in A\}$ is an external network for $\overline{A}$.

Let $\Gamma$ be a set. We put $\Sigma(\Gamma) = \{x\in\er^\Gamma:\;|\{\gamma\in\Gamma:\;x(\gamma)\neq 0\}|\leq\omega\}$. A compact space $K$ is \emph{Corson compact} if there is a homeomorphic embedding of $K$ into $\Sigma(\Gamma)$ for some set $\Gamma$.

\begin{definition}Let $X, Y$ be sets, $\O\subset\exp(X)$ closed under countable increasing unions, $\N\subset\exp(Y)$ and $f:\O\to\N$. We say that $f$ is \emph{$\omega$-monotone} if
\begin{enumerate}[\upshape (i)]
	\item $f(A)$ is countable for every countable $A\in\O$;
	\item if $A\subset B$ and $A,B\in\O$ then $f(A)\subset f(B)$;
	\item if $\{A_n\setsep n\in\omega\}\subset \O$ and $A_n\subset A_{n+1}$ for every $n\in\omega$ then $f(\bigcup_{n\in\omega}A_n) = \bigcup_{n\in\omega}f(A_n)$.
\end{enumerate}
\end{definition}

\begin{definition}A space $T$ is \emph{monotonically Sokolov} if we can assign to any countable family $\F$ of closed subsets of $T$ a continuous retraction $r_\F:T\to T$ and a countable external network $\N(\F)$ for $r_\F(T)$ in $T$ such that $r_\F(F)\subset F$ for every $F\in\F$ and the assignment $\N$ is $\omega$-monotone.
\end{definition}

\section{proofs of the main results}
The following proposition is the key tool to prove Theorem \ref{tMain}. The idea of the proof is moreover in following the lines of the proof Lemma 2.4 (a) of \cite{kubisSmall}. In order to obtain the $\omega$-monotonicity, we use a fixed ``Skolem function'' (see e.g. \cite[Section 2]{cuthKalenda}) to construct the elementary submodels from \cite[Lemma 2.4]{kubisSmall}.

\begin{proposition}\label{pMain}
Let $K$ be a Corson compact space. Then, to any countable family $\F$ of closed subsets of $K$ we can assign a countable set $M(\F)\subset K$ and a retraction $r_\F$ such that
  \begin{enumerate}[\upshape (i)]
    \item $r_\F(F)\subset F$ for every $F\in\F$,
    \item $r_\F(K) = \overline{M(\F)}$ and
    \item the assignment $\F\mapsto M(\F)$ is $\omega$-monotone.
  \end{enumerate}
\end{proposition}
\begin{proof}In the proof we denote by $\B$ the set of all the rational open intervals in $\er$. Without loss of generality we may assume that $K\subset\Sigma(\Gamma)$ for some set $\Gamma$. If $\gamma_1,\ldots,\gamma_n\in\Gamma$ and $U_1,\ldots,U_n\in\B$, we put
$$[\gamma_1,\gamma_2,\ldots,\gamma_n;U_1,U_2,\ldots,U_n] = \{x\in\er^\Gamma:\;x(\gamma_i)\in U_i\text{ for any }i=1,2,\ldots,n\}.$$
For $S\subset K$ we denote by $\supp(S)$ the set of all $\gamma\in\Gamma$ such that $s(\gamma)\neq 0$ for some $s\in S$. Note that $\supp(S)$ is countable whenever $S\subset K$ is countable. For $x\in K$ and $A\subset\Gamma$ we denote by $x\restriction_A$ the point in $\Sigma(\Gamma)$ defined by $x\restriction_A(\gamma) = x(\gamma)$ for $\gamma\in A$ and $x\restriction_A(\gamma) = 0$ for $\gamma\in\Gamma\setminus A$.

If $F$ is a non-empty closed subset of $K$, then we pick a point $x_F\in F$. For any $k\in\en$, $\gamma_1,\ldots,\gamma_k\in\Gamma$ and $U_1,\ldots,U_k\in\B$ we pick, if it exists, $x(F,\gamma_1,\ldots,\gamma_k;U_1,\ldots,U_k)\in F\cap[\gamma_1,\ldots,\gamma_k;U_1,\ldots,U_k]$.

Take a countable family $\F$ of closed subsets of $K$. We will recursively construct $M(\F)$. Let $M_0(\F) = \{x_F\setsep F\in\F\cup\{K\}\}$. Assume that $n\in\omega$ and we have countable sets $M_0(\F),\ldots,M_n(\F)$. Let
\begin{align*}
M_{n+1}(\F) = M_n(\F)\; \cup\; \bigcup_{F\in\F}\Bigl\{x(F,\gamma_1,\ldots,\gamma_k;U_1,\ldots,U_k)\setsep & \gamma_1,\ldots,\gamma_k\in \supp(M_n(\F)),\\ & U_1,\ldots,U_k\in\B, k\in\en\Bigr\}.
\end{align*}
Notice that $M_{n+1}(\F)\subset K$ is countable since the set $M_n(\F)$ is countable. We will prove that $M(\F) = \bigcup\{M_n(\F)\setsep n\in\omega\}$ and $r_\F(x) = x\restriction_{\supp(M(\F))},\; x\in K$ are as promised.

\begin{claim}$r_\F(F) \subset\overline{F\cap M(\F)}$ for every $F\in\F\cup\{K\}$
\end{claim}
\begin{proof} Take an arbitrary $F\in\F\cup\{K\}$, $x\in F$ and $W\in\tau(r_\F(x),[0,1]^\Gamma)$. There are $j\in\en$, $\gamma_1,\ldots,\gamma_j\in\Gamma$ and $U_1,\ldots,U_j\in\B$ such that $V = [\gamma_1,\ldots,\gamma_j;U_1,\ldots,U_j] \subset W$ and $r_\F(x)\in V$. It suffices to find some $y\in F\cap M(\F)\cap V$.

If $\{\gamma_1,\ldots,\gamma_j\}\cap \supp (M(\F)) = \emptyset$ then we put $y = x_F$. It is immediate that $y\in F\cap M(\F)$. Moreover, since $y\in M(\F)$, $y(\gamma) = 0$ for $\gamma\in\Gamma\setminus\supp(M(\F))\supset\{\gamma_1,\ldots,\gamma_j\}$; hence, $y(\gamma_i) = 0 = r_\F(x)(\gamma_i)$ for every $i\in\{1,\ldots,j\}$. Thus, $y\in V$.

Otherwise, find $k\in\en$ and $i_1,\ldots,i_k$ such that
$$\{\gamma_{i_1},\ldots,\gamma_{i_k}\} = \{\gamma_1,\ldots,\gamma_j\}\cap \supp (M(\F)).$$
Now it is enough to put $y = x(F,\gamma_{i_1},\ldots,\gamma_{i_k},U_{i_1},\ldots,U_{i_k})$ and observe that then $y\in F\cap M(\F)\cap V$.
\end{proof}
From the claim above it immediately follows that $r_\F:K\to K$ is a continuous retraction, $r_\F(K)\subset\overline{M(\F)}$ and $r_\F(F)\subset F$ for every $F\in\F$. Notice, that whenever $x\in M(\F)$, $\supp(x)\subset\supp(M(\F))$ and hence $r_\F(x) = x$. Consequently, $M(\F)\subset r_\F(K)$ and $r_\F(K) = \overline{M(\F)}$.
\begin{claim}The assignment $\F\mapsto M(\F)$ is $\omega$-monotone.
\end{claim}
\begin{proof}
It takes a straightforward induction to see that the set $M(\F)$ is countable for any countable family $\F$ of closed subsets of $K$ and the assignments $\F\mapsto M_n(\F)$ are $\omega$-monotone for every $n\in\omega$. Now it is easy to observe, e. g. by \cite[Proposition 4.3]{tkaRoj}, that the assignment $\F\mapsto M(\F) = \bigcup\{M_n(\F)\setsep n\in\omega\}$ is $\omega$-monotone.\end{proof}
\end{proof}
\begin{proof}[Proof of Theorem \ref{tMain}]It is immediate that (i)$\Rightarrow$(ii). Suppose that $K$ is monotonically $\omega$-monolithic. It follows from \cite[Corollary 2.2]{gruen} that $K$ must be Corson; hence, e.g. by \cite[Lemma 1.6]{kalendaSurvey}, it has a countable tightness. By \cite[Theorem 2.10]{tka13}, any monotonically $\omega$-monolithic space of countable tightness is monotonically monolithic; hence, we proved (ii)$\Rightarrow$(i). The implication (iii)$\Rightarrow$(ii) is proved in \cite[Proposition 4.4]{tkaRoj}.

Finally, suppose that $K$ is monotonically $\omega$-monolithic. By \cite[Corollary 2.2]{gruen}, $K$ is Corson. Now we can apply Proposition \ref{pMain} to convince ourselves that for any countable family $\F$ of closed subsets of $K$ we can choose a countable set $M(\F)\subset K$ and a retraction $r_\F$ such that $r_\F(K) = \overline{M(\F)}$, $r_\F(F)\subset F$ for any $F\in\F$ and the assignment $M$ is $\omega$-monotone. Since $K$ is monotonically $\omega$-monolithic, to each countable set $S\subset K$ we can assign a countable family $\O(S)\subset\exp(K)$ which is an external network of $\overline{S}$ in such a way that $\O$ is $\omega$-monotone. Let $\N(\F) = \O(M(\F))$. Then $\N(\F)$ is a countable external network of $r_\F(K) = \overline{M(\F)}$ in $K$ and the assignment $\N$ is $\omega$-monotone because it is a composition of $\omega$-monotone mappings. Hence, $K$  is monotonically Sokolov and (ii)$\Rightarrow$(iii) follows.
\end{proof}
\begin{proof}[Proof of Corollary \ref{cMain1}]This is an easy consequence of Theorem \ref{tMain} because every Collins-Roscoe space is monotonically monolithic; see e.g. \cite[Lemma 3.1]{gruen}.
\end{proof}
\begin{proof}[Proof of Corollary \ref{cMain2}]By \cite[Example 3.12]{tka13}, there exists a compact Collins-Roscoe space which is not Gul$'$ko. By Corollary \ref{cMain1}, every compact Collins-Roscoe is monotonically Sokolov.
\end{proof}

\section*{Acknowledgements}

The author would like to thank O. Kalenda for pointing out the paper \cite{tkaRoj} which initiated this work.

\def\cprime{$'$}


\begin{thebibliography}{10}

\bibitem{cuth}
Marek C{\'u}th.
\newblock Simultaneous projectional skeletons.
\newblock {\em J. Math. Anal. Appl.}, 411(1):19--29, 2014.

\bibitem{cuthKalenda}
Marek C\'uth and Ond{\v{r}}ej F.~K. Kalenda.
\newblock Rich families and elementary submodels.
\newblock accepted in Cent. Eur. J. Math. (2014), preprint avaiable at
  http://arxiv.org/abs/1308.1818.

\bibitem{gruen}
Gary Gruenhage.
\newblock Monotonically monolithic spaces, {C}orson compacts, and {$D$}-spaces.
\newblock {\em Topology Appl.}, 159(6):1559--1564, 2012.

\bibitem{gul}
S.~P. Gul{\cprime}ko.
\newblock The structure of spaces of continuous functions and their hereditary
  paracompactness.
\newblock {\em Uspekhi Mat. Nauk}, 34(6(210)):33--40, 1979.

\bibitem{kalendaSurvey}
Ond{\v{r}}ej F.~K. Kalenda.
\newblock Valdivia compact spaces in topology and {B}anach space theory.
\newblock {\em Extracta Math.}, 15(1):1--85, 2000.

\bibitem{kubisSmall}
Wies{\l}aw Kubi{\'s} and Henryk Michalewski.
\newblock Small {V}aldivia compact spaces.
\newblock {\em Topology Appl.}, 153(14):2560--2573, 2006.

\bibitem{tkaRoj}
R.~Rojas-Hern{\'a}ndez and V.~V. Tkachuk.
\newblock A monotone version of the {S}okolov property and monotone
  retractability in function spaces.
\newblock {\em J. Math. Anal. Appl.}, 412(1):125--137, 2014.

\bibitem{tka05}
Vladimir~V. Tkachuk.
\newblock A nice class extracted from {$C_p$}-theory.
\newblock {\em Comment. Math. Univ. Carolin.}, 46(3):503--513, 2005.

\bibitem{tka07}
Vladimir~V. Tkachuk.
\newblock Condensing function spaces into {$\Sigma$}-products of real lines.
\newblock {\em Houston J. Math.}, 33(1):209--228 (electronic), 2007.

\bibitem{tka13}
Vladimir~V. Tkachuk.
\newblock Lifting the {C}ollins-{R}oscoe property by condensations.
\newblock {\em Topology Proc.}, 42:1--15, 2013.

\end{thebibliography}
\end{document}